\documentclass[12pt,reqno]{amsart}

\usepackage{times}
\usepackage[T1]{fontenc}
\usepackage{mathrsfs}
\usepackage{latexsym}
\usepackage[dvips]{graphics}
\usepackage[titletoc,title]{appendix}
\usepackage{epsfig}

\usepackage{amsmath,amsfonts,amsthm,amssymb,amscd}
\input amssym.def
\input amssym.tex
\usepackage{color}
\usepackage{hyperref}

\usepackage{color}
\usepackage{breakurl}
\usepackage{comment}
\newcommand{\bburl}[1]{\textcolor{blue}{\url{#1}}}

\newcommand{\be}{\begin{equation}}
\newcommand{\ee}{\end{equation}}
\newcommand{\bea}{\begin{eqnarray}}
\newcommand{\eea}{\end{eqnarray}}

\newtheorem{thm}{Theorem}[section]

\newtheorem{cor}[thm]{Corollary}

\newtheorem{que}[thm]{Question}

\numberwithin{equation}{section}

\begin{document}

\title{Additive complements for two given asymptotic densities}
\author{H\`ung Vi\d{\^e}t Chu}
\email{\textcolor{blue}{\href{mailto:hungchu2@illinois.edu}{hungchu2@illinois.edu}}}
\address{Department of Mathematics, University of Illinois at Urbana-Champaign, Urbana, IL 61820, USA}

\begin{abstract}
Let $0\le \alpha \le \beta\le 1$. 
For any finite set $B\subset\mathbb{N}$, we show that there exists a set $A\subset\mathbb{N}$ such that $\underline{d}(A+B) = \alpha$ and $\overline{d}(A+B) = \beta$, where $\underline{d}(A+ B)$ and $\overline{d}(A+B)$ are the lower and upper asymptotic densities of the set $A+B$, respectively. This partially answers a question by Faisant et al. A theorem involving the so-called highly sparse sets was proved in the previous arXiv version of this note; however, as pointed out by Sai Teja Somu, the proof of the theorem was flawed. The theorem is now an open question. 
\end{abstract}



\maketitle

\section{Introduction}

A popular problem in the field of additive combinatorics is to construct a set satisfying certain density requirements. We use the asymptotic density defined as follows: given any set $A\subset \mathbb{N}$ of natural numbers, the lower asymptotic density $\underline{d}(A)$ and the upper asymptotic density $\overline{d}(A)$ are 
\begin{align*}
    \underline{d}(A) \ =\ \liminf\limits_{n\rightarrow \infty} \frac{\# A\cap [1, n]}{n} \mbox{ and } \overline{d}(A) \ =\ \limsup\limits_{n\rightarrow \infty} \frac{\# A\cap [1, n]}{n},
\end{align*}
respectively. If $\underline{d}(A) = \overline{d}(A)$, then we say $A$ has the asymptotic density, denoted by $d(A)$, and $d(A) = \underline{d}(A) = \overline{d}(A)$. Faisant et al.\ \cite{FGPS} considered the following question. 
\begin{que}\normalfont\label{q1}
Let $B\subset \mathbb{N}$ be a set with $d(B) = 0$ and a fixed real number $0<\alpha <1$, is there always a set $A\subset \mathbb{N}$ such that $d(A+B) = \alpha$? If not, what additional condition(s) should we put on $B$ so that such a set $A$ exists?
\end{que}
Faisant et al.\ \cite{FGPS} proved that if $B\subseteq \mathbb{N}$ is finite, there exists $A\subseteq \mathbb{N}$ with $d(A + B) = \alpha$ for any $\alpha\in [0,1]$. Then such a set $A$ is called an additive complement of $B$ with respect to density $\alpha$. (Note that our definition of additive complements is different from the one used by Chen and Fang \cite{CF}.) Recently, Leonetti and Tringali \cite{LT} generalized the result by considering quansi-densities and showing the existence of $A$ when $B$ is finite. 

In this note, we investigate the situation when two possibly distinct density values $0\le \alpha\le  \beta\le 1$ are given. Questions involving two given density values are also popular in the literature. For example, Bienvenu and Hennecart \cite{BH} solved the question of the existence of a set $A\subset \mathbb{N}$ with $d(A) = \alpha$ and $d(A+A) = \beta$. Hegyv\'{a}ri et al. \cite{HHP} proved the existence of a set $A\subset\mathbb{N}$ such that $d(A) = 0$, $\underline{d}(A\cdot A) = \alpha$, and $\overline{d}(A\cdot A)=\beta$. As we shall see later, we borrow the technique used by Hegyv\'{a}ri et al. in \cite[Theorem 4.4]{HHP} to partially answer the following question. 

\begin{que}\normalfont\label{q2}
Let $B\subset \mathbb{N}$ be a set with $d(B) = 0$ and fixed real numbers $0\le \alpha \le \beta \le 1$, is there always a set $A\subset \mathbb{N}$ such that $\underline{d}(A+B) = \alpha$ and $\overline{d}(A+B) = \beta$? If not, what condition(s) should we put on $B$ so that such a set $A$ exists?
\end{que}

\begin{thm}\label{m3}
Let $B\subset \mathbb{N}$ with $d(B) = 0$ and fixed real numbers $0\le \alpha\le \beta\le 1$. If there exists a set $C\subset\mathbb{N}$ such that $d(C+B) = \beta$, then there exists a set $A\subset \mathbb{N}$ such that $\underline{d}(A+B) = \alpha$ and $\overline{d}(A+B) = \beta$.
\end{thm}

The following corollary answers a part of Question Q5 in \cite{FGPS}. The corollary is immediate from the result by Faisant et al.\ and Theorem \ref{m3}.

\begin{cor}
Let $0\le \alpha\le \beta \le 1$. If $B$ is finite, then there exists a set $A\subset\mathbb{N}$ such that $\underline{d}(A+B) = \alpha$ and $\overline{d}(A+B) = \beta$.
\end{cor}

Given a set $A\subset \mathbb{N}$, we let $A(n)$ denote $\# A\cap [1, n]$, the cardinality of $A\cap [1, n]$.

\section{Proof of Theorem \ref{m3}}
We borrow the technique used in \cite[Theorem 4.4]{HHP} by creating a sequence of subsets of $C$. However, because our original set $C$ does not necessarily have $d(C) = 0$ as the original set $Q$ in \cite[Theorem 4.4]{HHP}, we exploit the fact that $d(B) = 0$ instead.

If $\alpha = \beta$, then we are done by setting $A = C$. Assume that $\alpha < \beta$. Let $(\gamma_k)_{k=1}^\infty$ be a strictly decreasing sequence such that $0< \gamma_k < \beta - \alpha$ and $\lim_{k\rightarrow\infty}\gamma_k \rightarrow 0$.

We shall form a sequence of subsets of $C$, called $(C_k)_{k=1}^\infty$, and an increasing sequence of positive integers $(n_k)_{k=1}^\infty$. In particular, $C_{k+1}$ is formed from $C_k$ by dropping integers in the range $[n_k+1, n_{k+1}]$. Let $C_1 = C$ and $n_1 = 1$. Since $d(C + B) = \beta$, we can choose $N$ such that for all $n \ge N$, we have $(C_1+B)(n) \ge (\alpha+\gamma_1) n$. We will build $(C_k)_{k=1}^\infty$ recursively such that $(C_k+B)(n)\ge (\alpha + \gamma_k) n$ for all $n\ge N$. 

Recursion: suppose that we already define $C_k$ and $n_k$. Then $C_{k+1}$ and $n_{k+1}$ are defined according to the parity of $k$. 

If $k$ is odd, let $n_{k+1} > n_k$ be the smallest positive integer such that there exists $n \ge N$ and 
$$\frac{(C_k\backslash [n_k+1, n_{k+1}] + B)(n)}{n} \ <\ \alpha + \gamma_{k+1}.$$ To see that $n_{k+1}$ is well-defined, we consider the function 
$$f(u) \ :=\ \frac{(C_k\backslash [n_k+1, u] + B)(u)}{u} \ =\ \frac{(C_k(n_k) + B)(u)}{u}.$$
Since $C_k(n_k)$ is finite and $d(B) = 0$, it follows that $\lim_{u\rightarrow\infty} f(u) = 0$ and so, there exists $u_0 > \max\{N, n_k\}$ verifying $f(u_0) < \alpha + \gamma_{k+1}$. Hence, $n_{k+1}\le u_0$. Set $C_{k+1} = C_k\backslash [n_{k}+1, n_{k+1}-1]$.

If $k$ is even, let $C_{k+1} = C_k$ and $n_{k+1} > n_k$ be the smallest positive integer for which $$\frac{(C_k+B)(n_{k+1})}{n_{k+1}} \ >\ \beta - 1/k.$$
To show that $n_{k+1}$ is well-defined, it suffices to show that $d(C_k+B) = \beta$. Note that $C_k\backslash [1, n_k] = C\backslash [1, n_k]$. Hence, for $n > n_k$, 
$$\frac{(C+B)(n)}{n} \ \ge\ \frac{(C_k + B)(n)}{n} \ \ge\ \frac{(C+B)(n) - ([1, n_k] + B)(n)}{n}.$$
Since $d(B) = 0$ and $d(C+B) = \beta$, we know that as $n\rightarrow\infty$, the left and right sides go to $\beta$, and so $d(C_k+B) = \beta$.

Set $A = \cap_{k=1}^\infty C_k$. Note that $A(n_k) = C_k(n_k)$. We show that $\underline{d}(A+B) = \alpha$ and $\overline{d}(A+B) = \beta$.

\begin{enumerate}
    \item Property 1: $\underline{d}(A+B) \ge \alpha$. 
    Indeed, let $n\ge N$ and choose $n_k \ge n$. By construction, we have
    $$(A+B)(n) \ =\ (C_k + B)(n) \ \ge\ (\alpha+\gamma_k) n,$$
    which implies Property 1 since $n$ is arbitrary.
    \item Property 2: $\underline{d}(A+B) \le \alpha$. Let $k$ be odd. By construction, there exists $n > N$ such that 
    \begin{align} \label{e111} (C_{k+1}\backslash \{n_{k+1}\} + B)(n)  \ < \ (\alpha +\gamma_{k+1}) n.\end{align}
    We claim that $n > n_k$; otherwise, $n\le n_k$, along with \eqref{e111}, would imply that $(C_{k} + B)(n)  <  (\alpha + \gamma_k) n$, which contradicts our construction of $C_k$. We have
    \begin{align*}
        &(A+B)(n) - (C_{k+1}\backslash \{n_{k+1}\} + B)(n)\\
        \ \le\ &(A+B)(n) - (A\backslash \{n_{k+1}\} + B)(n) \ \le\ B(n).
    \end{align*}
    Therefore, $(A+B)(n) \le B(n) + (\alpha+\gamma_{k+1}) n = n(B(n)/n + \alpha + \gamma_{k+1})$. As $k\rightarrow\infty$, $n\rightarrow\infty$ since $n > n_k$, which implies Property 2. 
    \item Property 3: $\overline{d}(A+B) \ge \beta$. Let $k$ be even. By construction, we have $$(A+B)(n_{k+1})  \ =\  (C_{k+1}+B)(n_{k+1}) \ \ge\ (\beta-1/k)n_{k+1}.$$
\end{enumerate}
From Properties 1, 2, 3 and the fact that $\overline{d}(A+B)\le d(C+B) = \beta$, we conclude that $\underline{d}(A+B) = \alpha$ and $\overline{d}(A+B) = \beta$, as desired.

\begin{que}\normalfont
A set $B = \{b_1, b_2, \ldots\}$ is said to be highly sparse if $B$ is either finite or infinite with $b_{j+1}/b_j \rightarrow \infty$ as $j\rightarrow \infty$.
Let $\alpha \in [0,1]$. 
For a highly sparse set $B$, is it true that there always exists $A\subset\mathbb{N}$ such that $d(A+B) = \alpha$? Theorem 1.5 in the previous arXiv version attempts to answer this question in the positive, but the bound for $\#U$ is wrong as it only holds when $b_1\neq b_2$.
\end{que}

\section*{Acknowledgement}
The author is thankful to Kevin Ford for fruitful discussions and to the anonymous referee for helpful comments. 
%%%%%%%%%%%%%%%%%%%%%%%%%%%%%%%%%%%%%%%%%%%%%%%%%%%%%%%%%%%%%%%%%%%%%%%%%%%%%%%%%%%%%%%%%%%%%%%%%%%%%%%%%%%%%%%%%%%%%%%%%%%%%%%%%%%%%%%%%%%%%%%%%%%%%%%%%%%%%%%%%%%%%%%%%%%%%%%%%%%%%%%%%%%%%%%%%%%%%%%%%%%%%%%%%%%%%%%%%%%%%%%%%%%%%%%%%%%%%%%%%%%%%%%%%%%%%%%%%%%%%%%%%%%%%%%%%%%%%%%%%%%%%%%%%%%%%%%%%%%%%%%%%%%%%%%%%%%%%%%%%%%%%%%%%%%%%%%%%%%%%%%%%%%%%%%%%%%%%%%%%%%%%%%%%%%%%%%%%%%%%%%%%%%%%%%%%%%%%%%%%%%%%%%%%%%

\end{document}